\documentclass[dvipdfmx]{article}
\usepackage{amsmath}
\usepackage{amsfonts}
\usepackage{latexsym}
\usepackage{mathtools}
\usepackage{amsthm}
\usepackage{amssymb}
\usepackage{tikz}
\numberwithin{equation}{section}
\newcommand{\cat}{\text{cat}}

\newcommand{\R}{\mathbb{R}}
\newcommand{\Z}{\mathbb{Z}}

\newcommand{\grad}{\text{grad}}
\newtheorem{theorem}{Theorem}
\newtheorem{theorem*}{Theorem}
\newtheorem{lemma}{Lemma}
\newtheorem{definition}{Definition}
\newtheorem{proposition}{Proposition}
\newtheorem{corollary}{Corollary}
\usepackage{docmute}

\title{An extension of Lusternik-Schnirelmann category of closed 1-form to non compact manifolds}
\author{Fukushi Kenji}

\begin{document}
\maketitle

\begin{abstract}
\noindent
Michael Farber introduced the Lusternik-Schnirelmann category $\cat(M,\xi)$ for the pair of finite CW complex $M$ and first-order cohomology $\xi$.
   It is inspired by the Morse-Novikov theory, which is a closed 1-form version of the Morse theory.
   An important result of this theory is that if the number of zeros of a closed 1-form $\omega$ on a closed manifold is less than $\cat(M,[\omega])$, then any gradient flows of $\omega$ has at least one homoclinic cycle.
   This paper begins with an explanation of Lusternik-Schnirelmann theory of closed 1-form, and extends Farber's results to general non-compact manifolds.
   We will also explain that Farber's results hold equally well on non compact manifolds, and explain the new phenomena related gradient flows of $\omega$ that occurs on non compact manifolds.
\end{abstract}

\section{Introduction}
The behavior of a function on a manifold is highly dependent on the topological properties of a manifold. Morse theory and Lusternik-Schnirelmann category (LS category) theory are both theories that clarify the relationship between the topology of a manifold and the properties of a function regarding its critical points.
An important consequence is that it gives a topological lower bound on the number of critical points for smooth Morse functions and more general functions on manifolds.

Morse theory is a theory that uses Morse functions to study the topology of manifolds.
The Morse function is a function whose critical points are all non-degenerate, and the most important result in Morse theory is that  the CW complex which is in the same homotopy type as the manifold can be constructed from the information of critical points.
Also, as a colorally of this reslut, the $k$-th critical points of the Morse function exist greater than or equal to the $k$-th Betti number of the manifold. In other words, the Betti number is the lower limit of the number of critical points of the Morse function, and this is called the Morse inequality.

Sergei Novikov generalized Morse theory by using closed 1-forms and their zeros instead of Morse functions. This is because a closed 1-form can be thought of as a derivative of a multivalued function, and if a certain homology theory can be constructed from the zeros of the closed 1-form, a generalization of the Morse inequality to multivalued functions can be obtained.
Novikov obtained a homology theory of local system coefficients (Morse-Novikov theory) whose generator is the zeros of a certain closed 1-form (Morse 1-form) . $\omega$ is called Morse 1-form if all zeros have a neighborhood such that $\omega$ coincides with the derivation of some Morse function. The Betti number of this homology gives the lower bound of the zero point of the Morse 1-form.

LS category theory considers an integer $\cat(X)$ of topologocal space $X$ which is greater than or equal to 0. $\cat(X)$ is homotopy-invariant.
$\cat(X)$ is the minimum number of null-homotopic open sets in $X$ such that are necessary to cover $X$. The most important result of this theory is that a smooth function on a closed manifold always has critical points greater than or equal to $\cat(X)+1$.
The important technique of Morse theory and LS category theory is the deformation of manifolds by the flow along some vector field called the gradient of a function, and both provide a lower limit on the number of critical points, and because their methods and results are similar, they are often compared.

Michael Farber extended the LS category theory based on Morse-Novikov theory.
He introduced the invariant $\cat(M,\xi)$ which gives an integer greater than or equal to zero for a pair of a finite CW complex and first-order cohomology class on it.
Whereas the Morse-Novikov theory gives a lower bound on the zeros of the Morse 1-form,
This category does not give a lower bound on the zeros of closed 1-forms of the same cohomology class. As we will see later, this is because, from the perspective of dynamical systems, it deals with a dynamical system that is different from Morse theory, LS category theory, and Morse-Novikov theory.
Therefore, in the LS category theory of closed 1-forms, there is the following new relationship between the zeros of closed 1-forms and the dynamical system of a certain flow. The flow is generated by a vector field determined from $\omega$ and the Riemannian metric, and is a generalization of the gradient of the function.

\begin{theorem*}
Let $M$ be a closed manifold and $\omega$ be a closed 1-form on it. If the number of zeros of $\omega$ is less than $\cat(M,[\omega])$, then $\omega$ metric the gradient flow of has at least one homoclinic cycle for any Riemannian metric.
\end{theorem*}
A homoclinic cycle is a collection of flow trajectories that connect critical points and move the critical points periodically, and is a very interesting trajectory from a dynamical system perspective.

In this paper, we extend this closed 1-form LS category to non compact manifolds. As a main result, we show that the gradient flow of $\omega$ has a homoclinic cycle when the following is satisfied.

\begin{itemize}
   \item The number of zeros of $\omega$ is less than $\cat(X,\omega)$
   \item $\grad(\omega)$ is a complete vector field, and there exists $c>0$ such that $|\grad(\omega)|>c$ outside some neighborhood of the zeros.
\end{itemize}

We will also confirm that in the case of compact manifolds, it does not matter how to take the Riemann metric that gives the flow, but in the case of non-compact manifolds, how to take the Riemann metric becomes important.
In particular, for non-compact manifolds, there are cases where a homoclinic cycle does not exist even if the conditions regarding the zero point are satisfied. In this case, a phenomenon occurs that does not occur in the compact case, where the flow escapes to infinity.
We will also give an example where $\grad(\omega)$ is complete when combined with the curvature condition.

\tableofcontents

\section{An extension of LS category to closed 1-form}

\subsection{Lusternik-Schnirelmann category of closed 1-form}

Morse theory is a theory that studies the topology of a manifold using gradient flows of Morse functions, and the LS category is a theory that also uses gradient flows of functions to connect the dynamical property of gradient such as the number of critical points and the topological complexity. In both cases, gradient flows of functions play a central role.

Although the flow becomes complicated when extending Morse theory to a closed 1-form, we constructed a homology theory using an easy-to-handle flow by using an appropriate Riemannian metric. This dynamical system is a Morse-Smale System ([12]), which is a dynamical system that has many properties similar to a gradient system of functions.
A Morse-Smale System is a dynamical system in which there is hyperbolicity at the fixed points of the system and transversality between invariant manifolds. In such systems, there are many results similar to gradient systems. In other words, Morse theory extended to closed 1-form has in common with Morse theory in that it is a Morse-Smale system, although the dynamical system it deals with is not a gradient system.

Extending the LS category with respect to closed 1-forms means studying the gradient flow of general closed 1-forms. Such flows no longer have anything in common with the Morse-Smale system.
For this reason, a new phenomenon related to dynamical systems exists, that is the existence of homoclinic cycles.

A homoclinic cycle is a family of heteroclinic orbits that connect fixed points of the gradient flow of $\{\gamma_1\dots \gamma_n \}$. For $1\leq i \leq n-1$.

\[\lim_{t\rightarrow \infty} \gamma_i(t) = \lim_{t\rightarrow -\infty} \gamma_{i+1}(t) \]
and \[\lim_{t\rightarrow \infty}\gamma_n(t) = \lim_{t\rightarrow -\infty} \gamma_1(t) \]
If $n=1$, it is especially called a homoclinic orbit.

\begin{theorem}
If the number of zeros of $\omega$on a closed manifold $M$ is less than the generalized invariant $\cat(M,[\omega])$ of $\cat(M)$, then any gradient flows of $\omega$ has at least one homoclinic cycle for any Riemannian metric.
\end{theorem}

Homoclinic cycles are very important orbits that have many applications in Hamilton dynamical systems and chaos theory.
It is clear that this trajectory does not exist in any gradient systems of functions.
The new invariant $\cat(M,[\omega])$ is no longer the quantity that gives the lower bound of the fixed points of flow. On the other hand, if the zero point of the closed 1-form is less than $\cat(M,[\omega])$, a homoclinc cycle exists. This study clarifies the relationship between such new dynamical property and topology.

\subsection{Definition of $\cat(M,[\omega])$}

Michae Farber extended the Lusternik-Schnirelmann category to closed 1-form. Farber defined an invariant of integers greater than or equal to 0 for the finite CW complex and first-order cohomology $[\omega]$ above it. If $[\omega]=0$, then $\cat(M,0)=\cat(M)+1$, which means $\cat(M,[\omega])$ is the generalization of the usual LS category.
To define $\cat(M,[\omega])$ for a general CW complex, we need to define a generalized differential form which is called continuous closed 1-form. A general differential form on a differentiable manifold is a smooth section into cotangent bundle. But continuous closed 1-form is defined as a section into a subsheaf of the sheaf of continuous functions.
The reason for extending the differential form to CW complexes is that defining them on CW complexes is easier to handle in terms of homotopy theory, and allows line integrals along non-smooth paths.

In this paper, we do not deal with CW complexes that are not smooth manifolds, so we do not introduce continuous closed 1-forms.

\begin{definition}\label{defls}
For a closed manifold $M$ and a first-order cohomology [$\omega$] on it, let $\cat(M,[\omega])$ be the smallest integer $k\geq 0$ such that satisfies the following. It is called the LS category of $\omega$.
\begin{itemize}
   \item For any positive integer $N$, there exist null-homotopic open sets $U_i, 1\leq i \leq k$ and open set $U$ on $M$ such that $M=U_1\cup\dots \cup U_k \cup U$
   \item (Conditions for $U$) There is a homotopy $h_t:U\times [0,1]\rightarrow M$ such that $h_0=id$, for any $x\in U$
   \[\int_{\gamma(x)}\omega \leq -N \]
   Where $\gamma(x)$ is a path along the homotopy $h_t$.
\end{itemize}
\end{definition}

$\cat(M,0)$ matches the normal LS category. This is because the line integral along the gradinet $\int _{\gamma}\omega := df = f(\gamma(1))-f(\gamma(0))$ is bounded on any closed manifolds, so
$U=\emptyset$ for sufficiently large $N$. In order to generalize $\cat(M,\omega)$ to non-compact manifolds later, common properties such as homotopy invariance will be discussed later.

When a manifold is deformed by using a closed 1-form negative gradient flow, it is divided into a part where the flow passes near the equilibrium point and a part where the integral value continues to decrease. The reason why the integral value continues to decrease is that when $\omega$ is integrated, it has a negative cycle, which did not exist in the gradient system.

If $M$ is deformed along the gradient flow $\phi_t$ between times $0\leq t \leq t$, then if we compare the definition and the above picture, we get a null-homotopic opening. The set $U_i$ corresponds to the parts that gather at the equilibrium point, and $U$ corresponds to the parts that cycle along the gradient.
If the number of equilibrium points is less than $\cat(M,[\omega])$, the equilibrium points and $U_i$ do not directly correspond. This is because if it were to correspond, the categorical coverage of the number of sheets less than $\cat(M,[\omega])$ would be removed. In other words, the parts that gather at the fixed point are not necessarily contractible, and this phenomenon leads to the existence of homoclinic cycles.

\subsection{Evaluation of the number of zeros of $\omega$ and $\cat(M,[\omega])$}

A smooth function on a closed manifold $M$ always has at least $cat(M)+1$ critical points. For the general closed 1-form $\omega$ rather than the exact form $df$, the generalized $cat(M,[\omega])$ is no longer a lower bound at the zeros of $\omega$. However, if $\omega$ has zero points less than $cat(M,[\omega])$, the negative flow of $\omega$ has an interesting trajectory in terms of dynamical system theory called a homoclinic cycle.
This does not depend on how to take a Riemannian metric for compact manifolds.

On a closed manifold $M$, for any $\xi \neq 0\in H^1(M,\Z)$, there is a closed 1-form with at most one zero point such that $[\omega]=\xi$.
Also, for any $n$, there exists a pair of CW complex and continuous 1-form such that $\cat(M,[\omega])=n, [\omega]\neq 0$. From these two results, we can see that there actually exists a set of $(M,[\omega])$ in which the number of zeros of $\omega$ is less than $\cat(M,[\omega])$.

\begin{theorem}For any closed manifold $M$ and $\xi \neq 0 \in H^{1}(M)$, there is a closed 1-form $\omega$ with at most one zero point and $[\omega]=\xi$.
\end{theorem}

\begin{theorem}\label{thmhatannon}
   For any $n>0$, there is a pair of $(M,\omega)$ such that $n=cat(M,[\omega])$ 
\end{theorem}

\begin{proof}
   Let the finite CW complex $Y$ be $X=Y\vee S^1$. Let $\xi\in H^1(X,\R)$ satisfy $\xi|_{Y}=0,\xi|_{S^1}\neq 0$. At this time
   \begin{eqnarray*}
     \cat(X,\xi) = \cat(Y)
   \end{eqnarray*}
holds true. From this, we can see that the category of closed 1-form can take any value.

   Let $X=U\cup U_1 \cup \dots \cup U_k $ be a categorical covering over $N$. If $U^{\prime}=U\cap Y,U^{\prime}_i=U_i\cap Y$, then
   $Y=U^{\prime}\cup U_1^{\prime} \dots \cup U^{\prime}_k$ is the covering of $Y$, and using this covering $\cat(X,\ xi)\geq \cat(Y)$.

   $U_i^{'}$ is null-homotopic in $Y$. This can be seen from the existence of retraction from $Y\vee S^1$ to $Y$.
   By showing that $U^{'}$ is also null-homotopic in $Y$, $\cat(X,\xi)\geq \cat(Y)$ is proven.
   $h_t:U\rightarrow Let X$ be a homotopy whose integral value is less than or equal to $-N$. Let $\tilde{X}$ be a covering that solves the fundamental group of $S^1$. In other words, it is a space in which each integer on the covering $\R$ of $S^1$ has a copy of $Y$. Equate $Y_0\subset \tilde{X}$ and $Y$.
   Let ratraction $r:\tilde{X}\rightarrow Y$ be the retraction that collects each $Y_n(n\neq 0)$ to the origin of $Y_0$.
   $h_t:U^{\prime}\rightarrow X$ lifts to $\tilde{h}_t:U^{\prime}\rightarrow Y_0 \subset \tilde{X}$. Because $\tilde{h}_t$ causes $U^{\prime}$ to move homotopically to a different sheet from $Y_0$ in $\tilde{X}$,
   $r\circ \tilde{h}_t$ is the transformation from $U^{'}$ to the base point of $Y$. Therefore, $U^{'}$ is collapsible in $Y$. Therefore, $\cat(X,\xi)\geq \cat(Y)$.

   $\cat(X,\xi) \leq \cat(Y)$ is shown. Let $Y=V_0\cup \cdots \cup V_r$ be the categorical open cover of $Y$. Since it is clear that $Y$ is reducible, we set it as $r\geq 1$.
   Let $U=S^1\cup V_0,U_i=V_i (1 \leq i \leq r)$. At this time, $V_0$ is taken so that it does not include the base point of the wedge sum.
   Transform $V_0$ homotopically to a single point in $Y$. Then, if we move $Y$ and $S^1$ to the base point where they are wedged and then rotate around $S^1$ by homotopy, we can calculate the integral from $\xi|_{S^1}\neq0$ You can lower the value as much as you like.
   Therefore, $\cat(X,\xi)\leq\cat(Y)$ holds true.
\end{proof}

The former theorem is a dynamical system result of the number of zeros in $\omega$,
The latter theorem is a result regarding the homotopy invariant $\cat(X, \xi)$.
Combining these basic results, we can conclude that there exists a closed 1-form that has zeros less than $\cat(M,[\omega])$.

\subsection{Evaluation of $cat(X,[\omega])$ by cup product}
$\cat(M,[\omega])$ has similar results in LS category. It is an evaluation of the lower bound of $\cat(M,[\omega])$ using the length of the cup product.

A well-known method for evaluating the lower bound of a classical LS category is one that uses the cup product.

\begin{theorem}Suppose the topological space $M$ has cup product length $n$. Cup product length $n$ means that there are cohomology elements $v_i\in H^{*}(M) (*>0, 1\leq i \leq n)$ such that $v_1 \cup v_2 \dots \cup v_n \neq 0$. Then $\cat(M) \geq n+1$
\end{theorem}

An important technique to prove this theorem is to define $v_i\in H^{*}(M)$ as a relative cohomology with a contractible open set $U_i$ $v_i\in H^{*}(M,U_i)$ The purpose was to lift it up and take the cup product. In the case of a category with closed 1-form, in addition to the contractible parts, there are also parts where the integral value decreases, and it is necessary to interpret these parts geometrically.

Hereafter, $[\omega]\in H^1(M,\Z)$, $\pi_1(M)/\ker(\omega)\cong \Z$. In this case, there exists $M$, a cyclic cover $\tilde{M}$, and a smooth function $f$ on $\tilde{M}$, and for the projection $p$ there is a smooth function $p^{*}\omega =df$.

Since the deck transformation group is an infinite cyclic group, we fix its generator $\tau$ to one and use the compact set $K\subset \tilde{M}$ as $\bigcup \tau^i (K)=\tilde{M Fix the set that is }$.

Take one field $k$ below.

\begin{definition}
   The homology class $z\in H_q(\tilde{M};k)$ is movable to $+\infty$ if $\cup_{i>N}\tau^i$ 
\end{definition}

This definition does not depend on how $K$ is chosen. Also, according to the following lemma,
In $\bigcup_{i>N}\tau^i(K)$ for $N$ where $z\in H_q(\tilde{M};k)$ is sufficiently large

When expressed as a cycle, it becomes movable to $+\infty$.

\begin{lemma}
   For phase space $M$, there exists an integer $N$ that satisfies the following.

   (a) If the cycle $c$ in $K$ is homologous to the cycle in $\cup_{i>N}\tau^i(K)$, it is movable to $+\infty$

   (b) If cycle $c$ in $K$ is homologous to the cycle in $\cup_{i<-N}\tau^i(K)$, movable to $-\infty$

\end{lemma}

When the negative gradient flow of $\omega$ at $M$ is pulled back to $\tilde{M}$, it becomes a negative gradient flow regarding $f$ from $p^{*}\omega=df$. The part where the integral continues to decrease due to the flow of $M$ continues to go through cycles in $M$, so the integral value continues to decrease. If we pull this trajectory back toward the covering, we get a trajectory that continues to lower the value of $f$ while lowering the sheet of covering. Therefore, the cycle in which the integral value continues to decrease in flow in $M$ corresponds to the cycle that satisfies movable to $-\infty$ in the covering. This correspondence is the reason for introducing movable to $\pm\infty$. Furthermore, cycles that satisfy movable to $\pm\infty$ have certain algebraic properties.

\begin{theorem}
   The following is equivalent to $z\in H_q(\tilde{M};k)$
   \begin{itemize}
     \item $z$ is movable to $+\infty$
     \item $z$ is movable to $-\infty$
     \item $z$ is an element of Torsion of $k[\tau,\tau^{-1}]$ module $H_q(\tilde{M};k)$
   \end{itemize}
\end{theorem}

\begin{proof}

   [6] See Theorem 10.20.

\end{proof}

$a\in k$ defines a local system as it appeared in the theory of Novikov homology. it is
The fiber is a local system of $k$,

\[\pi_1(M) \rightarrow \text{Aut}(k):g \mapsto a^{<[\omega],[g]>} \]
It was a local system corresponding to the monodromy expression. Isomorphism for tensor product of local systems for $a,b\in k^{*}$
\[a^{\xi}\otimes b^{\xi} \cong (ab)^{\xi} \]
exists, which makes the cup product

\[\cup:H^{p}(M,a^{\xi})\otimes H^{q}(M,b^{\xi}) \rightarrow H^{p+q}(M,( ab)^{\xi}) \]

is determined.

\begin{theorem}
   Meets the following
   \[u\in H^p(M,a^{\xi}) \quad v\in H^q(M,b^{\xi}) \quad w_j \in H^{p_j}(M,k )\quad (1\leq j \leq r) \]
   If exists $\cat(M,\xi)>r$

   \begin{itemize}
     \item $p,q,p_j>0$
     \item $u \cup v \cup w_1 \cup \cdots \cup w_r \in H^{p+q+p_1+\cdots +p_r}(M,(ab)^{[\omega]})$ is $0$ Not
     \item $a,b\in k^{*}$ is not included in Supp($M,[\omega]$)

   \end{itemize}
\end{theorem}

Supp($M,[\omega]$) is defined as follows.
$\tau:\tilde{M}\rightarrow \tilde{M}$ induces homomorphism between homologies. Therefore, $H_{*}(\tilde{M};k)$ is a $\Lambda:=k[\tau,\tau^{-1}]$ module.
Since $k[\tau,\tau^{-1}]$ is a Noetherian ring,
$H_{*}(\tilde{M};k)$ and its Torsion part

\[\textstyle{ T = \text{Tor}_{\Lambda}(H_{*}(\tilde{M;k})) }\]
is also a finitely generated $\Lambda$ module and a finite-dimensional $k$ vector space. This Torsion part corresponds to movable to $\pm \infty$.

The operation of multiplying $\tau$ is a $k-$ linear mapping of $T\rightarrow T$, and since there is an inverse mapping that multiplies $\tau^{-1}$, it becomes an isomorphism. Let Supp$(M,[\omega])\subset k^{*}$ be the set of all eigenvalues of this isomorphism. This is a finite set.

   Once you indicate the following Lifting Property, all you have to do is consider the cup product as in the case of the normal LS category. First, we show the following theorem.

   \begin{theorem}
     Let $M$ be a closed manifold and $[\omega]\in H^1(M,\Z)$. For $a\notin$\,Supp$(M,[\omega])$,
     Suppose there is a compact subset $K$ of $M$ such that $[\omega]\mid_{K}=0$. For any $F\subset K$ that satisfies the 
      for sufficiently large $N$

     \[H^p(M,F;a^{[\omega]}) \rightarrow H^p(M;a^{[\omega]}) \]

     is a surjection.

     This $N$ can be $N$ of Lemma .

     \begin{proof}
       If $[\omega]=0$, $F=\emptyset$, so the theorem holds.
       Therefore, $[\omega]\neq 0$.

       Let $f:\tilde{M}\rightarrow \R$ be a function that satisfies $p^{*}\omega = df, f(\tau x)= f(x)+1$. From $[\omega]\mid_K=0$, $K$ lifts to $\tilde{M}$. Fix one such lift and shift $f$ by a constant factor so that $f\mid_K$ has a value at $[0,c]$. Let $c\in \Z$.

       Take $N^{\prime}$ such that it satisfies the lemma for $K\subset \tilde{M}$. In other words, any cycle in $K\subset \tilde{M}$ is homologous to the cycle in $\bigcup_{j \leq -N^{\prime}} \tau^j K$. In this case, the cycle becomes the twist source of $H_{*}(\tilde{M};k)$.

       If $F\subset K$ satisfies the definition  for $N>N^{\prime}$, then

       \[ H_{*}(\tilde{F};k) \rightarrow H_{*}(\tilde{M};k) \]
The image becomes the source of the twist. Since $\tilde{F}$ is an infinite number of copies of $F$, it becomes $H_{*}(\tilde{F};k) \cong H_{*}(F;k)\otimes_k \Lambda$.

       The theoremasserts that if we consider a complete sequence of space pairs,

       \[ H^{*}(M;a^{[\omega]}) \rightarrow H^{*}(F;a^{[\omega]}) \]
This is equivalent to the fact that becomes a zero map. Considering the duality with homology, this

       \[ H_{*}(F;a^{-[\omega]}) \rightarrow H_{*}(M;a^{-[\omega]}) \]
This is equivalent to the fact that becomes a zero map. Since $[\omega]\mid_F=0$, $a^{[\omega]}\mid_F \cong k$. Therefore

       \[ H_{*}(F;k) \rightarrow H_{*}(M;a^{-[\omega]}) \]
It suffices to show that is a zero map.

       inclusion $F\rightarrow M$ is the same as $F \xrightarrow{\subset} \tilde{M} \xrightarrow{p} M$.
       We know that $\tilde{F}$ is movable to the image of the mapping above $-\infty$ is included in $\Lambda-$torsion. Therefore
       \begin{eqnarray}
        \textstyle{p_{*}(T) = 0 ,\quad T=\text{Tor}_{\Lambda}(H_{*}(\tilde{M};k))}
      \end{eqnarray}
       If we show that, we have shown the theorem. well-known complete sequence

       \[ \cdots \rightarrow H_i(\tilde{M};k) \xrightarrow{\tau-b} H_i(\tilde{M};k) \xrightarrow{p_{*}} H_i(M;b^{[\omega]}) \rightarrow \cdots \]
think of. Here $b=a^{-1}$. $\tau-b:T\rightarrow T$ is isomorphic from $a\notin \text{Supp}(M,[\omega]) $. Therefore, the set of torsion elements of $H_{*}(\tilde{M};k)$ is the same as the image of $\tau-b$, so from the complete sequence above, $p_{*}(T)=0$
\end{proof}

     (Proof of theorem)

     You can set $[\omega]\neq 0$. Assume $\cat(M,[\omega])\leq r$. In this case, it becomes a contradiction to show that the $r+2$ cup product, as assumed in the theorem, is always 0. $M$ can be divided into two parts such that $[\omega]\in H^{1}(X;\Z)$ is trivial.

     \[ M=U\cup V,\quad [\omega]\mid_{\overline{U}}=0,\quad [\omega]\mid_{\overline{V}}=0 \]

     There exists an open set $U,V$ such that For $\overline{U},\overline{V}$, take a sufficiently large $N$ that satisfies the lemma . From $\cat(M,[\omega]) \leq r$, the open coverage $M=F\cup F_1\cup \cdots \cup F_r$ and the integral value is uniform for any $N>0$ There is a homotopy that changes $F$ so that it is less than or equal to $-N$. Then, for $F\cap U$, $u\in H^q(M;a^{[\omega]})$ is $H^q(M,F\cap U Lift to ;a^{[\omega]})$. The same goes for $F\cap V$. From the result of LS category, the element of $w_j\in H^{d_j}(M;k)$ is also lifted to $H^{d_j}(M;F_j;k)$, so

     \[ \tilde{u}\cup \tilde{v}\cup \tilde{w}_1 \cup \cdots \cup \tilde{w}_r \in H^{*}(M,M;(ab)^ {[\omega]}) = 0 \]

     It becomes. Therefore, it must be $\cat(M,[\omega])>r$.

   \end{theorem}

\section{Generalization of $\cat(M,[\omega])$ to non compact manifolds}

In this section, we generalize $\cat(M,[\omega])$ introduced by Farber to non-compact manifolds. The reason $\cat(M,[\omega])$ was defined only for compact manifolds and finite CW complexes is that for only compact manifolds or CW complexes $M$, $\cat(M,[\omega])$ does not depend on how to take a equivalence class of $[\omega]$. $\cat(X,\omega)$, which will be introduced as follow, gives integers for manifold $X$ and closed 1-form $\omega$, and even if they are in the same cohomology class of closed 1-forms The value can be different depending on how the closed 1-form is taken. Therefore, $[\quad]$ of $[\omega]$ is removed and it is written as $\cat(X,\omega)$.
This difference in $\cat(X,\omega)$ within the same cohomology class is due to the difference in behavior at infinity, which does not occur in a compact manifold.
Therefore, if $\omega$ and $\omega^{\prime}$ are the same cohomology on a compact manifold, $\cat(X,\omega) = \cat(X,\omega^{\prime})$.

Let the compact manifold be written as $M$ and the non compact manifold as $X$.

\subsection{Definition of $\cat(X,\omega)$}

\begin{definition}For a manifold $X$ and a closed 1-form $\omega$, let $\cat(X,\omega)$ be the smallest integer $k\geq 0$ that satisfies the following. 

\begin{itemize}
   \item \text{(1)} For any positive integer $N$, there exists an open cover of $X$, $X=U \cup U_1 \cup \dots \cup U_k$ such that satisfies the following.

   \item \text{(2)} Each $U_i$ is null-homotopic in $X$.

   \item \text{(3)} There is smooth homotopy $h_t:U\times [0,1] \rightarrow X$ such that $h_0$ is inclusion and for any $x\in U$, $\int _{\gamma_x}\omega \leq -N$. Where $\gamma_x$ represents the line integral along the homotopy.

\end{itemize}

\end{definition}

\begin{proposition}
   For any bounded $C^{\infty}$ function $f$ on closed 1-form $\omega$ and $X$, $cat(X,\omega)=cat(X,\omega + df)$
\end{proposition}

\begin{proof}
    For any line integral over $\gamma$, if $|f|\leq C$
    \begin{eqnarray*}
      \left| \int_{\gamma}\omega - \int_{\gamma}\omega + df \right| = |f(\gamma(1)) - f(\gamma(0))| \leq 2C
    \end{eqnarray*}

    From this, open coverage and homotopy for $\omega$ that satisfy the conditions for $N^{\prime}\geq N + 2C$ also satisfy the conditions for $N$ of $\omega + df$. 
    Therefore $cat(X,\omega + df)\leq cat(X,\omega)$. The reverse is also true, so $cat(X,\omega) = cat(X,\omega+df)$
\end{proof}

In particular, the following facts are clear.

\begin{corollary}
   If $\omega,\omega^{\prime}\in \Omega^1_c(X)$ is $[\omega]=[\omega^{\prime}] \ \text{in} \ H_c ^1(X)$, then $\cat(X,\omega)=\cat(X,\omega^{\prime})$.
\end{corollary}

Since any $C^{\infty}$ function is bounded on a compact manifold, $cat(X,\omega)$ is an invariant that does not depend on how to take equivalence classes.
In the non compact case, $\omega,\omega+df$ have the same value for bounded function $f$,
In particular, it is an invariant for compact support cohomology.

\subsection{Homotopy invariance}

Discuss the homotopy invariance of $\cat(X,\omega)$.
We also prove the homotopy invariance of $\cat(M,[\omega])$ for the compact manifold which was skipped in the previous section.

In a compact manifold, it is invariant to any homotopy, but in the non-compact case, homotopy requires certain conditions for the integral over $\omega$.

\begin{lemma}\label{lemhomo}
Suppose that the manifolds $X$ and $Y$ are homotopy equivalent.
If the smooth homotopy maps $\phi:X\rightarrow Y,\psi:Y\rightarrow X$ satisfies the following, then $cat(X,\omega)=cat(Y,\psi^{*}\omega )$.

\begin{itemize}
   \item $\psi \circ \phi :X\rightarrow X$ and $\phi \circ \psi :Y\rightarrow Y$ each have a homotopy between them and the identity map. And the line integral of closed 1-form $\omega, \psi^{*}\omega$ along its homotopy starting from any point in $X,Y$ is bounded.
\end{itemize}
In particular, if $X,Y$ are compact, the conditions regarding line integrals are always satisfied.
So $cat(X,\omega)=cat(Y,\psi^{*}\omega)$.
\end{lemma}

\begin{proof}
   $\phi \circ \psi :X\rightarrow X$. Let $r_t$ be the homotopy between $\phi \circ \psi$ and the identity map. By assumption, the line integral along the homotopy is bounded for any $x\in X$, so there exists some $C>0$ such that $|\int_{r_t(x)} \omega| < C$ be. Assume $cat(Y,\psi^{*}\omega)\leq k$.
   Then, for any $N>0$, there exists an open neighborhood $V\cup V_1 \dots V_k $ of $Y$, which is the condition of $\cat$.
   Assume that there exists a homotopy such that the line integral of $\phi^{*}\omega$ is less than or equal to $-N-C$ with homotopy $h_t:V\rightarrow Y$. Define $X=U\cup U_1 \dots U_k$ as $U=\psi^{-1}(V),U_i=\psi^{-1}(V_i)$. $U_i$ are collapsible in $X$. Show that there exists a homotopy $h^{\prime}_t :U\rightarrow X$ such that the line integral is uniformly less than or equal to $-N$. Homotopy $h^{\prime}_t:U\rightarrow X$

   \begin{eqnarray}
     h^{\prime}_t =
     \begin{cases}
       r_{2t}(x) & (2t\leq1)\\
       \psi(h_{2t-1}(\phi(x))) & (2t\geq1)
     \end{cases}
   \end{eqnarray}
It is defined as Then at any $x\in U$

\begin{eqnarray}
   \int_{h^{\prime}_t(x)} \omega = \int_{r_{2t}(x)} \omega + \int_{h_{2t-1}(\phi(x))} \psi ^{*}\omega \leq -N
\end{eqnarray}
become. Since this is a homotopy that makes the line integral of $U$ uniformly less than $-N$, it becomes $cat(X,\omega)\leq \cat(Y,\psi^{*}\omega)$ Recognize. $\cat(X,\omega) \geq shows \cat(Y,\psi^{*}\omega)$. Let $\cat(X,\omega)\leq k$. In this case, similarly, the open coverage $X=U\cup U_1\dots U_k$, $U_i$ is contractible, and the line integral of $\omega$ is uniformly $-N$ with the homotopy $g_t:U\rightarrow X$. There is something that makes it less than $N-C$. If $V:=\psi^{-1}(U),V_i:=\psi^{-1}(U_i)$, then $V_i$ is contractible.
$\psi \circ \phi:Y \rightarrow$ Let $r^{\prime\prime}$ be the homotopy that connects $Y$ and the identity map.
Homotopy $h^{\prime\prime}:V\rightarrow Y$

\begin{eqnarray}
   h^{\prime \prime}_t :=\begin{cases}
   r^{\prime\prime}_{2t}(x) & (2t\leq1)\\
   \phi(g_{2t-1}(\psi(x))) & (2t\geq1)
\end{cases}
\end{eqnarray}
It is defined as In this case, the line integral along the homotopy is

\begin{eqnarray}
   \int_{h^{\prime\prime}_t(x)}\psi^{*}\omega &=& \int_{r^{\prime\prime}_t(x)}\psi^{*} \omega + \int_{g_{2t-1}(\psi(x))} \phi^{*}\psi^{*} \omega \\
  &=& \int_{r^{\prime\prime}_t(x)}\psi^{*}\omega + \int_{g_{2t-1}(\psi(x))} \omega \leq - N
\end{eqnarray}
and $V,V_1,\dots,V_k$ and $h^{\prime\prime}_t$ satisfy the condition when category $-N$, so $cat(X,\omega)\geq It becomes cat(Y,\psi^{*}\omega)$. From this, we know that $cat(X,\omega)=cat(Y,\psi^{*}\omega)$.
\end{proof}
It is expected that $\cat(X,\omega)$ will be extended to general CW complexes. In that case, it is more convenient to define it as a homotopy invariant using continuous homotopy.
For two manifolds, the existence of a smooth homotopy equivalence map is equivalent to the existence of a continuous homotopy equivalence mapping ([8]Proposition 17.8). In this way, there is no difference between continuous and smooth homotopy among manifolds, and since we do not extend it to general CW complexes in this paper, we define only smooth homotopy.

\begin{corollary}
   Let $X,Y$ be homotopy-equivalent manifolds with smooth proper homotopy.
   In other words, for a proper and smooth map $\phi:X\rightarrow Y,\psi:Y\rightarrow X$, $\phi\circ\psi, \psi\circ\phi$ are homotopic to the identity map and the homotopy is proper. Proper homotopic means that a proper map $F:X\times [0,1]\rightarrow X$, $\psi\circ$ connects $\phi\circ\psi$ and the identity map $id\mid_X$. Homotopy $G:Y\times [0,1]\rightarrow Y$ connecting $\phi$ and identity map $id\mid_Y$. This means that both times $[0,1]\rightarrow Y$ are proper. In this case, the support on $Y$ is a compact closed 1-form $\omega$

   \[ \cat(X,\phi^{*}\omega) = \cat(Y,\omega) \]
   holds true.

\end{corollary}

A continuous map is called proper if the inverse image of any compact set is also compact. The reason for considering a proper map is that while the closed 1-form $\omega$ on $Y$ has a compact support, $\phi^{*}\omega$ also has a closed 1-form $\omega$ on $X$, where the support is compact.

\begin{proof}
   For proper map $\phi.\psi$, find the proper homotopy connecting $\phi\circ\psi$ and $id_Y$.

   \[ F:Y\times [0,1] \rightarrow Y, F\mid_0 = id_Y,F\mid_1 = \phi\circ\psi \]
   Suppose that on $Y$

   \begin{eqnarray}\label{eqntan}
    \int_{\gamma_t(y)} \omega
  \end{eqnarray}
   Show that is uniformly bounded.

   For each point $y\in Y$ of $Y$, the line integral is

   \[\int^1_0 \omega\left( \frac{d\gamma_{y,t}}{dt} \right) dt \]
It is. Since $F$ is proper, $F^{-1}(\text{supp}(\omega))$ is a compact set of $Y\times [0,1]$. Next $F^{-1}(\text{supp}(\omega))\rightarrow \R$

\[ (x,t)\mapsto \omega\left( {\frac{d\gamma_{y,t}}{dt}} \right)
   \]
is bounded above because its domain is compact. Therefore, (\ref{eqntan}) is uniformly bounded on $Y$.

The same holds true for the homotopy connecting $X$, $id\mid_X$, and $\psi\circ\phi$, so from the lemma  $\cat(X,\phi^{*}\omega)=\cat(Y,\omega)$.

\end{proof}

\subsection{The exsistence of Homoclinic cycles}
\indent
$\cat(M,[\omega])$, extended to closed 1-form does not give a lower bound of the number of zeros of $\omega$. But if the number of zeros of $\omega$ is less than $cat(M,[\omega])$, then there is the following very interesting result regarding the dynamical system of the gradient flow.

\begin{theorem}\label{themhatanzawa}
   If the number of zeros of a closed 1-form $\omega$ on a closed manifold $M$ is less than $\cat(M,[\omega])$, then the any gradient flows of $\omega$ for any Riemannian metric has at least one homoclinic cycle.
\end{theorem}
See [6] $\S$10 for the original proof.

In the ordinary LS category, the number of critical points of a function is evaluated from $\cat(M)$ by using the following fact.  The fact is that any points in a closed manifold must converge to each critical point with a gradient flow and a closed manifold is collapsed along the gradient flow of the function.

In the closed 1-form version, there exists a closed 1-form with zeros less than $\cat(M,[\omega])$. We regard stable manifolds $W^s$ of zero point $p$ of $\omega$ as null-homotopic open covering.

\[W^s_p = \{x\in M| \lim_{t\rightarrow \infty}\phi_t(x) = p \} \]

It means that manifold is not compressible with the gradient flow of closed 1-form generally. This leads to the existence of homoclinic cycles on the gradient flow.

In the previous section, we extended $\cat(M,[\omega])$ to general non compact manifolds. The theorem of the existence of homoclinic cycles also generalizes to non compact manifolds.

The trajectory of a gradient flow on a compact manifold is divided into a part that converges to the zeros of $\omega$ and a part that continues to lower the integral value around the cycle.
The part in which the integral value continues to decrease is a part that does not exist in the case of the gradient of a function. In the non compact case, in addition to these two parts that appear in the gradient flow, there is also a part that decreases the integral value as it goes to infinity. On a non-compact manifold, it can be shown that a homoclinic cycle exists if all parts of the flow other than the part that converges to the fixed point reduce the integral value to $-\infty$.

The condition for Riemannian metricis is required in order for the integral value of the part where flow escapes to infinity to decrease to $-\infty$. This is because gradient is a vector field determined by $\omega$ and the metric. This was not necessary for compact manifolds. The fact that the behavior of a dynamical system changes by replacing the metric in this way is a new phenomenon in non compact manifolds.

\subsection{Proof of the Main Theorem}
Extending the existence theorem of homoclinic cycles to non compact manifolds. On non compact manifolds, conditions for the Riemannian metric are required in addition to the number of zeros of $\omega$.

\begin{theorem}[Main theorem]\label{thmmain}
   If a closed 1-form $\omega$ on a complete Riemannian manifold $(X,g)$ satisfies the following, then the gradient flow of $\omega$ has at least one homoclinic cycle.
\end{theorem}

\begin{itemize}
   \item $\omega$ has zeros less than $cat(X,\omega)$
   \item The vector field $\grad(\omega)$ is complete, and there exists a certain positive real number $c>0$ and a neighborhood $U$ of the zero point, and on $M-U$ $|grad(\omega)| \geq c$
\end{itemize}

\begin{proof}
   Let the set of zeros of $\omega$ be $\{p_1,\dots p_n\}$. Each $p_i$ has an open neighborhood $U$ such that $\omega|_{U} = df ,f(p)=0$ , $f$ if called Lyapunov function that decreases along gradient flow. The trajectory of the flow in $U$ of $x\in U$ is written as $C_{x}$. Construct a special neighborhood of the zero point $p$ using $f$.
\end{proof}

   \begin{lemma}\label{lemmahatano}
With a compact neighborhood $B$ such that $B\subset{U}$ and $\delta \geq 0$,
\begin{itemize}
   \item (i) For each $x\in U$, is $C_x\cap B$ empty or $C_x \cap f^{-1}([-\delta,\delta])$

   \item (ii)$x\in \partial B$ and $f(x)\in (-\delta,\delta)$ then $C_x\cap B \subset{\partial B}$

\end{itemize}
There exists $B$ such that
   \end{lemma}

   \begin{proof}
     Let $N_{2\epsilon}$ be the compact neighborhood included in $U$ near $2\epsilon$ of $p$. First, with $\delta >0$, $\text{dist}(C_x\cap f^{-1}([-\delta,\delta]) ,p)\leq \frac{\epsilon}{2}$ When

     \begin{eqnarray}
       C_x\cap f^{-1}([-\delta,\delta]) \subset{N_{\epsilon}}
     \end{eqnarray}
Show that there exists $\delta$ such that If such $\delta$ does not exist, then $x_k\in U$

\begin{eqnarray}
  x_k\in \left[\frac{-1}{k},\frac{1}{k}\right],\quad
  d(x_k,p)\leq \frac{\epsilon}{2}
\end{eqnarray}
There exists a point sequence $\{x_k\}$ such that Similarly, we can also take $t_k$ such that $d(x_k \cdot t_k,p)\geq \epsilon$ for $t_k$, and we can make it a definite sign by taking a subsequence of $t_k$. Assuming $t_k\geq 0$, we do not lose generality. That is, in the sequence $t_k>0$, $d(x_k,p)\geq\epsilon$ and $x \cdot [0,t_k] \subset{L^{-1}([-\frac{1}{k },\frac{1}{k}])}$ exists.

If we rearrange $x_k$ so that $d(x_k,p)=\frac{\epsilon}{2}$, then $N_{\frac{\epsilon}{2}}$ is more compact than some $x There exists a convergent subsequence that converges to \in N_{\frac{\epsilon}{2}}$.

Also, since $d(x_k,x_k\cdot t_k)\geq \frac{\epsilon}{2}$ and the velocity of flow is bounded in the vicinity of $p$, $t_k$ converges to 0. Never. If $\{t_k\}$ is a bounded sequence and is included in $(0,T]$ at some $T$, then a convergent subsequence also exists for $t_k$. From the continuity of flow, $x_k \cdot[0,t]$ converges to $x \cdot[0,t]$. This is
\begin{eqnarray}
\bigcap_{k\geq1}f^{-1}\left(\left[\frac{-1}{k},\frac{1}{k}\right]\right) = f^{-1} (0)
\end{eqnarray}
include. However, this contradicts the fact that $x\in N_{\epsilon} - \{p\}$ and that $f$ monotonically decreases with flow.

From this, $\{t_k\}$ is unbounded. It may be $t_k\geq1$. However, in this case, $x_k \cdot[0,1]$ also converges to $x \cdot[0,1]$, and $x \cdot [0,1] \subset{f^{-1}(0) }$, which is unreasonable. From these, we can see that the point sequence $\{x_k\}$ does not exist in the first place, so we know that $\delta>0$ that satisfies the condition exists.

The following closure of the open set satisfies the conditions of the lemma.

\begin{eqnarray}
   B := \overline{\{x\in f^{-1}((-\delta,\delta))|\text{dist}(C_x\cap f^{-1}(-\delta,\delta) , p) < \frac{\epsilon}{2} \}}
\end{eqnarray}

Condition (i) first considers $x\in B$ and $\{x_i\} \in \text{Int}(B)$ that converges to it. Any interval $[t_1,t_2] including $0$ that satisfies x \cdot[t_1,t_2]\subset{f^{-1}([-\delta,\delta])}$ and a small real number $ For \alpha>0$, a sufficiently large $i$ is
$x_i \cdot[t_1+\alpha,t_2-\alpha]\subset{f^{-1}((-\delta,\delta))}$ is satisfied. Therefore, $C_x\cap f^{-1}([-\delta,\delta])\subset{B}$. Conversely, from the continuity of flow, we know that $C_x\cap B\subset{f^{-1}([-\delta,\delta])}$. Therefore, (i) is satisfied.

Condition (ii) is that if a point $C_x$ is included in $\text{Int}(B)$, then $C_x\cap f^{-1}((-\delta,\delta))$ This can be seen from the fact that all are included in $\text{Int}(B)$.
   \end{proof}

(Return to proof of \ref{thmmain})

Suppose that the number of zero points $n$ is less than $cat(M,\omega)$. If there is no homoclinic cycle, then using the neighborhood of the lemma around the zero point, for any $N$ we have an open cover $M=U\cup U_1 \dots U_n$ and $U_i$ is reducible. , $U\times [0,1]\rightarrow We show that there exists a homotopy of X$ where the line integral is uniformly less than $-N$. This results in
$n\geq cat(M,\omega)$, leading to a contradiction.

For any zero point $p_i$ of $\omega$, there exists a compact neighborhood $B_i$ that satisfies the lemma. From the configuration of $B_i$, when the flow trajectory leaves $B_i$, it passes through $\partial _{-} B_i:=f^{-1}(-\delta)\cap B_i$.

\begin{lemma}\label{lemmahatano}
   For any integer $N>0$, there exists $\epsilon>0$ that satisfies the following.

   \begin{itemize}
   \item The flow that exited $p\in\partial_{-}B_{i,\epsilon}$ at time $t=0$ becomes $p \cdot t\in \text{Int }(B_{i,\epsilon})$, then the integral of that orbit becomes less than or equal to $-N$.

\end{itemize}

\end{lemma}

\begin{proof}
   For some $N>0$, any $\epsilon>0$ means that the neighborhood $B_{\epsilon}$ of $p_i$ does not satisfy the condition. Fix one $\epsilon_0$. Since the conditions of the lemma are not satisfied, there exists $p_{i,n}\in \partial_{-} B_i, t_{i,n}>0 ,s_{i,n}<0$ that satisfies the following conditions.
   
   \begin{itemize}
     \item (a) $p_{i,n} \cdot[s_{i,n},0] \subset{B_i}$ and $\text{dist}(p_{i,n} \cdot s_{i, n},p_i)< \frac{1}{n} $
     \item (b) $\text{dist}(p_{i,n} \cdot t_{i,n},p_i) < \frac{1}{n}$
     \item (c) $\int_{p_{i,n} \cdot[0,t_{i,n}]} \omega \geq -N $
   \end{itemize}

   $p_{i,n}$ has a convergent subsequence in $\partial_{-}B_i$. Let it be $q_i\in \partial_{-}B_i$. At this time, $t_{i,n} and s_{i,n}$ also converge to $\infty,-\infty$ for that subsequence, respectively. This is the equilibrium of the dynamical system of the flow where $x_{i,n} \cdot t_{i,n}$ and $x_{i,n} \cdot s_{i,n}$ are the zero points of $\omega$. Since it converges to a point, it can be seen from the uniqueness of the solution to an ordinary differential equation.

   Examine the $\omega-$ limit set and $\alpha-$ limit set of $q_i$. The $\alpha$ limit set is that flow always includes a trajectory in the negative direction in $B_i$, and that $p_{i,n} \cdot s_{i,n}$ converges to the equilibrium point. It can be seen that $p_i$.

   From the assumption about closed 1-form $\omega$, $|grad(\omega)|$ does not disappear outside the vicinity of the zero point and the lower bound of the velocity is positive, and $\lim_{t\ rightarrow \infty}\int_{q_i \cdot[0,t]} \omega \geq From -N$, flow must converge in the positive direction. Therefore, the $\omega-$ limit set is one of the equilibrium points of flow.

   If $\omega(q_i)=p_i$, the assumption is violated. Therefore, $\omega(q_i)=p_j (i\neq j)$. There exists $T_0$ such that $q_i \cdot [T_0,\infty) \in B_j$ for some minimum $T_0 >0$. For sufficiently large $n$, the flow trajectory of $p_{i,n}$ enters $B_j$ near $T_0$. The flow of $p_{i,n}$ reenters $B_i$ at $t\rightarrow +\infty$, so it leaves $B_j$ again at a certain time. Therefore
   $t_{i,n}$ becomes $p_{i,n}^{\prime}=p_{i,n} \cdot t_{i,n}^{\prime}\in \partial_{-}B_j ^{\prime} > T_0$ exists. By taking the subsequence of $p_{i,n}^{\prime}$ again, it converges to $q_i^{\prime}\in B_j$ with $p_{i,n}^{\prime}$ You can take it as you like. Since $t_{i,n}^{\prime} - T_0$ converges to $\infty$, $\alpha(q_i^{\prime}) \in B_j$. Since the equilibrium point is finite, repeating this operation will yield a homoclinic cycle connecting $q_i$. This means that the assumption that a neighborhood like the one in the lemma cannot be taken was wrong, so the lemma has been proven.

(Return to proof of \ref{thmmain})

For each zero point of $\omega$, take the neighborhood $B_i$ of the lemma.

\begin{eqnarray}
   U := \left\{x\in M \middle| \text{For some $t_x > 0$}, \int_{x \cdot[0,t_x]} \omega = -N \right\}
\end{eqnarray}

Then $U$ is an open set.
And for the zero point $\{p_1,\dots p_n \}$ of $\omega$

\begin{eqnarray}
   U_i := \left\{x\in M \middle| \text{For some $t_x>0$}, x \cdot t\in \text{Int}(B_i)\ \text{and} \ \int_{x \cdot [0,t_x]} \omega > - N\right\}
\end{eqnarray}
Then $U_i$ are also open sets. And these cover $M$. If $x\notin U$, then $\lim_{t\rightarrow \infty} \int_{x[0,t]}\omega \geq -N$. Therefore, this flow converges to an equilibrium point, and it can be seen that for a certain $i$, $x\in U_i$. Furthermore, if we consider the flow due to $\grad(\omega)$ to be a homotopy, this becomes a homotopy that uniformly makes the integral value less than or equal to $-N$.

Finally, if we show that $U_i$ is the displacement retract of $B_i$, we can show that $U_i$ is contractible and $cat(M,\omega) \leq n$, which is a contradiction. From this, $\grad(\omega)$ has a homoclinic cycle.

The function $[0,\infty )$ from $U_i$

\begin{eqnarray}
\phi_i(x) :=\min \{t \geq 0|x \cdot t \in B_i \} = \inf\{t\geq 0|x \cdot t\in \text{Int}(B_i) \}
\end{eqnarray}

It is defined as
The equal sign on the right side can be seen from the construction method of the neighborhood $B_i$ because $\partial_{-}B_i$ is the exit set of the flow.

If we show that $\phi_i$ is continuous, we can retract from $U_i$ to $B_i$, which is contractible, and the proof of the theorem is completed. For a sequence of $x\in U_i$ and $x_k \in U_i$ that converges to $x$,

\begin{eqnarray}
    \limsup \phi_i(x_k) \leq \phi_i(x) \leq \liminf \phi_i(x_k)
\end{eqnarray}
All you have to do is show.
Show the left-hand inequality. From the second definition of (\ref{a}), if $t>0$ is $x\cdot t\in \text{Int}(B_i)$, then $x_k \cdot for sufficiently large $k$ t \in \text{Int}(B_i)$. Therefore, any sufficiently large $k$ is $\phi_i(x_k)\leq t$, and $\limsup \phi_i(x_k) \leq t$. Since $t>\phi_i (x)$ is arbitrary, we know $\limsup \phi_i(x_k) \leq \phi_i (x)$. $\liminf \phi_i(x_k) \geq \phi_i(x)$ can also be shown using the first definition of (\ref{a}). Therefore, $\phi_i$ is continuous.

\end{proof}

The key to the proof is to find a neighborhood with good properties for each equilibrium point in the flow.
It is a neighborhood, which we showed by Lemma \ref{lemmhatano}, that exists such that when a flow leaves that neighborhood and enters the same neighborhood again, the line integral approaches $-\infty$ sufficiently.
By taking such a neighborhood, it is possible to collect the flow into a contractible neighborhood around the equilibrium point by allowing the flow to flow until the time when the line integral does not fall to a certain standard. If you extend the flow too much, it will end up passing near other equilibrium points.
This method of determining the neighborhood is based on the proof in [10].

When flow is a gradient system and $\omega=df$, a neighborhood where the above situation does not occur can be taken. In other words, once a flow exits, it can take a neighborhood where it never returns.
[6] uses such a neighborhood to prove the theorem \ref{themhatanzawa} by pulling the flow back to the covering space where $\omega$ is exact.

For non-compact manifolds, the Riemannian metric is not arbitrary. A condition is required for the behavior of $\grad(\omega)$ at infinity. It is the second condition of the theorem \ref{thmmain}. This second condition is always easy to satisfy for compact manifolds, so it is only necessary for non-compact manifolds.

An example of a flow that does not satisfy the second condition of theorem\ref{thmmain} is a flow where the line integral does not diverge while going to infinity. The gradient of a bounded function on $\R$, for example $f=\tan^{-1}(x)$, is

\begin{eqnarray}
  \grad(f) = \frac{1}{1+x^2}\frac{\partial}{\partial x}
\end{eqnarray}
becomes. differential equation

\begin{eqnarray}
  \dot{x} = \frac{1}{x^2+1}
\end{eqnarray}
is a differential equation satisfied by the gradient system of $f$ on $\R$. $\grad(f)$ is complete because the integral curve of $f$ and its arbitrary order differential is defined above bounded $\R$.
However, $\grad(f)$ disappears at infinity, and the integral over $\R$

\begin{eqnarray}
   \int^{\infty}_0 (\grad(f),\grad(f)) dx = \frac{\pi}{2}
\end{eqnarray}
Therefore, the integral value does not diverge. This gradient system is an example of not satisfying the second condition of theorem\ref{thmmain}.

\section{Dynamical system of the gradient flow of on non compact manifolds}

\subsection{The differences from compact manifolds}
In a compact manifold, the quantity $\cat(M,[\omega])$ is a homotopy invariant that does not depend on how to take the equivalence class of $[\omega]$ 
Any homotopy that lowers the integral value used in the definition of $\cat(M,[\omega])$ on a closed manifold could be used. However, if we extend this definition to a non compact state, $\cat(X,\omega)$ depends on the closed 1-form even if it is the same cohomology class.
This problem arises from the complexity of dynamical systems on non compact manifolds.

\begin{proposition}In real line $\R$, the derivations of the functions 0 and x are in same de-Rham cohomology class of $\R$. But $\cat(\mathbb{R},0)=1$, $\cat(\mathbb{R} ,dx)=0$
\end{proposition}

\begin{proof}
    If we integrate along a homotopy that moves in parallel in the negative direction, we can uniformly lower the integral value arbitrarily, so $U$ in the definition of category can be taken as $U=\mathbb{R}$.  Therefore $cat(\mathbb{R},dx)=0$. $cat(\mathbb{R},0)=1$ is obvious.
\end{proof}

As you can see from this example, the integral value decreases as the flow goes to infinity. In the case of non-compact manifolds, such trajectories are included in addition to continuing around the corresponding cycle.
Even if orbits have different properties in a dynamical system, if the integral value can be lowered, they can be combined into one $U$. Based on this consideration, in the previous chapter we described the non compact manifold $\cat(M,[\omega])$ was successfully extended. However, if $\omega$ is replaced, the part of the flow that goes to infinity will naturally be different, and as a result, there will be a difference in the value of category.
An example of calculating $\cat(X,\omega)$ is further described below.

Also, since the gradient is a vector field determined by $\omega$ and the Riemannian metric, it becomes important to take the Riemannian metric on non compact manifolds, which was not considered important until now.
We discuss the relationship between how to take the Riemannian metric and the existence of homoclinic cycles, which is a phenomenon specific to non compact manifolds.

\subsection{Some Calculations of $\cat(X,\omega)$}
For non compact manifolds, we can choose a pair of manifold and closed 1-form such that $cat(X,\omega)$ is arbitrarily large. This can be seen from the following lemma.

\begin{lemma}
Let $M$ be a smooth manifold, and let $\omega$ be a closed 1-form on $M$.
For any bounded $C^{\infty}$ function $f$ on $\mathbb{R}^n$
$cat(M\times \mathbb{R}^n,\omega + df ) = cat(M,\omega)$
\end{lemma}

\begin{proof}
   $\psi:M\times \mathbb{R}^n\twoheadrightarrow M$ is projection and $\phi:M\hookrightarrow M\times \{0\}$ is inclusion, these become  homotopy maps $M\times \R^n$ to $M$.
   The integral of $\omega+df$ along the homotopy connecting $\phi\circ\psi$ and the identity map is the difference between two points on the fiber of $\mathbb{R}^n$ of $f$. be. Since $f$ is bounded, this means that the line integral is uniformly bounded homotopy, so $\cat(M\times\mathbb{R}^n,\omega+df)=\cat(M,\omega)$
\end{proof}

This result is extended to the general case of vector bundles.

\begin{corollary}
   Let $E\rightarrow M$ be a vector bundle over a closed manifold $M$. If $\pi:E\rightarrow M$ is a projection and $s$ is a zero section, $s\circ \pi$ becomes an identity map and a homotopy map. When the following holds true, $\cat(E,\pi^{*}\omega)=\cat(M,\omega)$.
\[
     \text{If some $C>0$ exists and for any $x\in E$}
     \left| \int_{\gamma(x)}\omega \right| < C
   \]
   $\gamma$ is a homotopy on $E$ that connects $s\circ \pi$ and the identity map.
\end{corollary}

\begin{proof}
   Similar to the lemma, $s\circ \pi$ is a homotopy with bounded integral value, so $\cat(E,\pi^{*}\omega)=\cat(M,\omega)$.
\end{proof}

An example of $\omega$ on $E$ that satisfies the conditions of this lemma is the closed 1-form $\omega$ on $M$ that is pulled back by $\pi$. form $\pi^{*}\omega$, and the integral along the fiber is always 0.
Differential forms with compact supports along the fiber is also used in the construction of Thom classes.

\begin{eqnarray*}
   \Omega_{cv}:= \{\omega\in\Omega^{*}(E)|\pi:\text{Supp}(\omega)\rightarrow M \text{is proper} \}
\end{eqnarray*}

\subsection{Relationship between Riemannian metric and the dynamical system of gradient flow}

Investigate how the Riemannian metric is affected on the dynamical system the gradient flow. In particular, we show that depending on how the Riemannian metric is taken, homoclinic cycles do not occur even if the condition about the number of zeros of $\omega$ is satisfied.

\begin{proposition}
   Let $M$ be a compact manifold and $\omega$ a closed 1-form with only one zero point $p$ and $\cat(M,[\omega])\geq 2$.
   Let $f:M\rightarrow \R$ be a smooth function that satisfies the following.

   \begin{itemize}
     \item $f$ takes the minimum value $0$ at $p$
     \item $f^{-1}(0)=\{p\}$
   \end{itemize}

   Take the product Riemannian metric $g=g_M\oplus dt^2$ on $M\times (-1,1)$,
   Consider closed 1-form $\omega^{\prime}:=\omega + d(tf + \frac{t^3}{3})$. In this case, $\omega^{\prime}$ satisfies the condition about the number of zeros, but the gradient flow of $\omega$ with $g$ does not have a homoclinic cycle.
\end{proposition}

As you can see from the condition, $f$ in the proposition can be a constant function outside the neighborhood of $p$, so if we think about it locally, we can see that there is $f$ that satisfies the condition.

\begin{proof}
   $\cat(M\times (-1,1),\omega^{\prime})=\cat(M,[\omega])$.
   This is because the line integral of $\omega^{\prime}$ along the retract from $M\times (-1,1)$ to $M\times \{0\}$ is bounded.

   The zero point of $\omega^{\prime}$ is only $\{p\}\times \{0\}$, and $\cat(M\times(-1,1),\omega^{\ prime})\geq 2$, this closed 1-form satisfies the condition regarding the zero point. This is shown below.

   \begin{eqnarray}
   \omega^{\prime}=\omega + tdf + (f+t^2) dt
   \end{eqnarray}
   For,

   \begin{eqnarray}
   W := M - \{p\}
   \end{eqnarray}
   Then $\omega^{\prime}$ does not become 0 on $W\times (-1,1)$. This is because the $t$ component is not 0. On $\{p\} \times (-1,1)$

   \begin{eqnarray}
     \omega^{\prime} = \begin{cases}
     0 \quad (p\times \{0\}) \\
     t^2 dt \neq 0 \quad(\{p\}\times(0,1)\cup \{p\}\times(-1,0)) \quad\text{($p$ is the Critical point of $f$)}
   \end{cases}
   \end{eqnarray}
   becomes. In other words, on $p\times(-1,1)$, only $p\times\{0\}$ is the zero point of $\omega^{\prime}$. Therefore $\omega^{\prime}$ has only one zero point.

   At points other than the zero point, the second component of $M\times (-1,1)$ always decreases along the flow. This is when the gradient is displayed in coordinates.

   \begin{eqnarray}
     \grad(\omega^{\prime}) = \begin{pmatrix}
    g_{M,11}^{-1} & g_{M,12}^{-1} & \dots & g_{M,1n}^{-1} & 0 \\
    g_{M,21}^{-1} & \dots & \dots & g_{M,2n}^{-1} & 0 \\
    \vdots & \dots & \dots & \vdots & 0 \\
    g_{M,n1}^{-1} & g_{M,n2}^{-1} & \dots & g_{nn}^{-1} & 0 \\
    0 & 0 & 0 & 0 & 1
\end{pmatrix}
\begin{pmatrix}
     \omega_1 + t \frac{\partial f}{\partial x_1} \\
     \omega_2 + t \frac{\partial f}{\partial x_2} \\
     \vdots \\
     \omega_n + t \frac{\partial f}{\partial x_n} \\
     f + t^2
\end{pmatrix}
   \end{eqnarray}
   This can be seen from the following. $\omega=\sum^n_{i=1}\omega_i dx_i$.

   Since the second component always decreases outside the equilibrium point of flow, it can be seen that there is no homoclinic cycle that converges to $\{p\}\times \{0\}$.
\end{proof}

In order to make the line integral along the homotopy bounded, we consider $M \times (-1,1)$ instead of $M\times \R$. In this case, the product Riemann metric $g$ is not complete, and depending on the order of $f$, the flow protrudes from $(-1,1)$ and $\grad(\omega^{\prime})$ is also complete. It is possible that it is not. It is possible to make the metric $g$ conformally complete ([21]). In that case, $|\grad(\omega^{\prime})|$ approaches 0 around $M\times \{\pm 1\}$, so the assumption of the existence theorem of the homoclinic cycle is not satisfied.

\subsection{Conditions for maintaining homoclinic cycles with perturbations of the Riemannian metric}

We will discuss a certain kind of stability of Homoclinic cycles. In other words, it provides an answer to the question of whether the homoclinic cycle is maintained no matter how much the Riemannian metric is purterbed.

\begin{theorem}\label{thmhom}
   Let $(M,g)$ be a complete Riemannian manifold. Suppose that the closed 1-form $\omega$ on $M$ satisfies the following.
   \begin{itemize}
     \item The number of zeros of $\omega$ is less than $\cat(M,\omega)$
     \item $\grad(\omega)$ is complete, and there is $c>0$ such that $|\grad(\omega)|\geq c$ in outside of some neighborhood of the zeros.
   \end{itemize}
   
In this case, the gradient flow of $\omega$ with respect to $g$ has a homoclinic cycle.
Furthermore, the gradient flow of $\omega$ for the Riemannian metric $g^{\prime}$ that satisfies the following also has a homoclinic cycle.

\begin{itemize}
   \item $g^{\prime}$ is a complete metric
   \item $\grad(\omega)$ with respect to $g^{\prime}$ is a complete vector field
   \item there is a $C\geq 1$ such that
   \begin{eqnarray}
     \frac{g}{C} \leq g^{\prime} \leq Cg
   \end{eqnarray}
\end{itemize}
\end{theorem}

\begin{proof}
   The gradient flow of $\omega$ with respect to $g^{\prime}$ satisfies the condition of the existence theorem of homoclinic cycle as follows.
   \begin{itemize}
     \item Outside the neighborhood of the zeros of $\omega$, $|\grad(\omega)|_{g^{\prime}} \geq \frac{c}{C}$
     \item By assumption, $\grad(\omega)$ is complete
   \end{itemize}
   As a result, the gradient flow regarding $g^{\prime}$ also has a homoclinic cycle.
\end{proof}

When $(M,g,\omega)$ has a homoclinic cycle, the metric $g^{\prime}$ which is equivalent to $g$ and complete with $\grad(\omega)$ also has $(M, It was found that g^{\prime},\omega)$ also has a homoclinic cycle.

The condition that the measurements are the same is easy to understand, but
It is generally difficult to check whether a vector field is complete. This is because it is necessary to solve ordinary differential equations on a Riemannian manifold.
With reference to the case of Euclidean space, we examine the conditions under which $\grad(\omega)$ is complete in a specific situation.

\begin{theorem}\label{themequation}
   For $C^{\infty}$ vector field $X=\sum^n_{i=1}f_i(x)\frac{\partial}{\partial x_i}$ on $\R^n$, all Let $f_i$ be bounded for $i$. In this case, $X$ is complete.
\end{theorem}

\begin{proof}
   Since the vector field is smooth, it is locally Lipschitz, and the uniqueness of the solution to the differential equation is shown. Since each coefficient is bounded, a constant time solution can be extended from Peano's existence theorem at each point in $\R^n$. Therefore, $X$ is complete.
\end{proof}

I would like to apply this result to Riemannian manifolds. Analysis on Riemannian manifolds often imposes the following conditions.

\begin{definition}[bounded geometry]
   A connected complete Riemannian manifold $(M,g)$ is called bounded geometry when it satisfies the following conditions:
   \begin{itemize}
     \item The injectivity radius $r_0$ is positive
     \item For any $m$-th covariant differential of the Riemannian curvature tensor $R$, there is a constant $C_m$ such that $|\nabla^m R|<C_m$ holds at each point.
   \end{itemize}

\end{definition}

The positiveness of injectivity radius means that each point on the Riemannian manifold can have a geodesic coordinate with radius $r_0$. Conditions regarding curvature are used to evaluate the curvature tensor, which appears in the Taylor expansion of the metric. Bounded geometry is a class to which various results of analysis, such as results of Sobolev spaces, are extended ([22]).

\begin{proposition}
   Let $(M,g)$ be a bounded geometry and $\omega$ be a closed 1-form on $M$.
   $\grad(\omega)$ is complete when the following is satisfied.
   \begin{itemize}
     \item $|\grad(\omega)|$ is bounded over $M$. Let the upper bound be $C_0$.
   \end{itemize}
\end{proposition}

\begin{proof}
   Because $(M,g)$ is bounded geometry
   Each point $p$ in $M$ has geodetic coordinates ($U_p,\phi_p$). i.e.

   \[g_{ij}(p)=\delta_{ij} , \text{for any $i,j,k$} ,  \Gamma_{ij}^k(p) = 0 \]
Where $\Gamma_{ij}^k$ is a Christoffel symbol.

   The positivety of the injectivity radius and the boundedness of $|\grad(\omega)|$ are required for the following lemma.

   \begin{lemma}
     There are certain $r,\epsilon>0$ such that $R,\epsilon$ satisfy the following.

     \begin{itemize}
       \item The disk around the radius $r$ of any point $p$ in $M$ is included in the geodetic coordinates of $p$.
       \item The flow $\phi(t)$ generated by $-\grad(\omega)$ whose initial value is $p$ is $\phi(t)$ if $t\in(-\epsilon,\epsilon) t)\in B_r(p)$
     \end{itemize}

     \begin{proof}
       Since $(M,g)$ is a bounded geometry, geodetic coordinates can be secured around any point $p$ with a uniform radius of $r_0$.
       Due to the boundedness of $|\grad(\omega)|$, there exists a certain $\epsilon$, and during the interval $(-\epsilon,\epsilon)$, the flow is $B_{r_0}(p)$. I know that he won't leave.
     \end{proof}

   \end{lemma}

   $\grad(\omega)$

   \begin{eqnarray}
     \grad(\omega) = \sum^n_{i=1} a_i \frac{\partial}{\partial x_i}
   \end{eqnarray}
   Suppose that Formula for Taylor expansion of metric

   \begin{eqnarray}
     g_{ij}(x) = \delta_{ij} - \sum R_{ikjl}x^kx^l + \sum_{[\alpha]\geq 3}(\partial^{\alpha}g_{ij}) (p)\frac{x^{\alpha}}{\alpha !}
   \end{eqnarray}
From the curvature assumption called bounded geometry, $\grad(\omega)$ is independent of $p$ even for Euclidean metrics on geodesic coordinates of radius $r_0$ around any point $p$. It can be suppressed by a constant. Therefore the differential equation

\begin{eqnarray}
   \dot{x} = \grad(\omega),\quad x(0)=p
\end{eqnarray}
Since the vector field is smooth, it becomes locally Lipschitz, and there exists $\epsilon>0$ that does not depend on $p$, and there exists a unique solution between $(-\epsilon,\epsilon)$ (theorem \ref{ theme}).
Therefore, $\grad(\omega)$ becomes a complete vector field.

\end{proof}

This result is the result of generalizing the theorem \ref{themequation} to Riemannian manifolds.

Combined with the theorem \ref{thmhom}, if $|\grad(\omega)|$ is bounded, the homoclinic cycle is maintained even if the metric is perturbed within the range where the bounded geometry is maintained and the metric is equivalent.
On compact manifolds, a homoclinic cycle exists even if the number of zeros is less than $\cat(M,[\omega]$, but on non compact manifolds, conditions about metrics and curvature are required for the existence of homoclinic cycles. This is a new phenomenon related dynamical systems of gradient flow of $\grad(\omega)$ on non compact manifolds, and provides a new connection between closed 1-forms, dynamical systems, topology, and curvature.

\end{document}